\documentclass[11pt,
]{amsart}
\usepackage{
amssymb, graphicx
}
\newtheorem{theorem}{Theorem}
\newtheorem{lemma}{Lemma}
\newtheorem{remark}{Remark}

\newtheorem{problem}{Inverse Problem}

\theoremstyle{remark}

\date{\today}

\newcommand{\R}{{\bf R}}
\newcommand{\sym}{\hbox{Sym}}

\newcommand{\be}[1]{\begin{equation}\label{#1}} 
\newcommand{\ee}{\end{equation}} 

\newcommand{\N}{\mathbf{N}}

\newcommand{\p}{\partial}

\title[Mixed ray transform]{Mixed ray transform on simple
  $2$-dimensional Riemannian manifolds}

\author[M. V. de Hoop]{Maarten V. de Hoop}
\address{Simons Chair in Computational and
    Applied Mathematics and Earth Science, Rice University, Houston,
    TX 77005, USA (\tt{mdehoop@rice.edu}).}

\author[T. Saksala]{Teemu Saksala}
\address{Department of Computational and Applied Mathematics,
  Rice University, Houston, TX, 77005, USA
   (\tt{teemu.saksala@rice.edu})}
  
  \author[J. Zhai]{Jian Zhai}
\address{Department of Computational and Applied Mathematics,
  Rice University, Houston, TX, 77005, USA
  (\tt{jian.zhai@rice.edu}).}
 
\begin{document}

\begin{abstract}
We characterize the kernel of the mixed ray transform on simple
$2$-dimensional Riemannian manifolds, that is, on simple surfaces for
tensors of any order.
\end{abstract}

\maketitle 

\section{Introduction}

We provide a characterization of the kernel of the mixed ray transform
on simple $2$-dimensional Riemannian manifolds for tensors of any
order. The key application pertains to elastic \textit{qS}-wave
tomography \cite{chapman1992traveltime} in weakly anisotropic media.

We let $(M,g)$ be a smooth, compact, connected $2$-dimensional
Riemannian manifold with smooth boundary $\partial M$. We assume that
$(M,g)$ is simple, that is, $\p M$ is strictly convex with respect to
$g$ and $\exp_p:\exp_p^{-1}(M) \to M$ is a diffeomorphism for every $p
\in M$. We let $SM=\{(x,v)\in TM;\|v\|_g=1\}$ be the unit sphere
bundle. We use the notation $\nu$ for the outer unit normal vector
field to $\p M$. We write $\p_{in}(SM)=\{(x,v)\in SM;x\in \partial M,
\langle v,\nu\rangle_g \leq 0\}$ for the vector bundle of inward
pointing unit vectors on $\p M$. For $(x,v)\in SM$, $\gamma_{x,v}(t)$
is the geodesic starting from $x$ in direction $v$, and $\tau(x,v)$ is
the time when $\gamma_{x,v}$ exits $M$. Since $(M,g)$ is simple
$\tau(x,v)< \infty$ for all $(x,v)\in \p_{in}(SM)$ and the
\textit{exit time function} $\tau$ is smooth in $\p_{in}(SM)$
\cite[Section 4.1]{Shara}.

\medskip

We use the notation $S^k M$, $k \in \N,$ for the space of smooth
symmetric tensor fields on $M$. We also use the notation $S^k M \times
S^\ell M , \: k,\ell\geq 1$ for the space of smooth tensor fields that
are symmetric with respect to first $k$ and last
$\ell$ variables. \textit{The mixed ray transform} $L_{k,\ell}$ of a
tensor field $f\in S^k M \times S^\ell M$ is given by the formula
\begin{equation}
\label{eq:mixed_ray_trans}
L_{k,\ell}f(x,v)=\int_{0}^{\tau(x,v)}  f_{i_1,\ldots,i_kj_1,\ldots, j_\ell}(\gamma(t))\dot\gamma(t)^{i_1}\cdots\dot\gamma(t)^{i_k}\eta(t)^{j_1} \cdots\eta(t)^{j_\ell} \mathrm{d}t, \:\:\:\:\:\:(x,v)\in\p_{in}(SM), \: \gamma=\gamma_{x,v},
\end{equation}
where we used the summation convention, while $\eta(t)$ is some unit length
vector field on $\gamma$ that is parallel and perpendicular to
$\dot{\gamma}(t)$ and depends smoothly on $(x,v) \in \p_{in}(SM)$. We
note that the definition of the mixed ray transform is different in
higher dimensions, due to the freedom in the choice of $\eta$ (See
\cite[Section 7.2]{Shara}). We consider the choice of $\eta(t)$ and
the mapping properties of $L_{k,\ell}$ in dimension $2$.

We define two linear operators the images of which are
contained in the kernel of $L_{k,\ell}$. For a $(k\times
\ell)$-tensor, $f_{i_1,\ldots,i_kj_1,\ldots,j_\ell}$, we introduce the
symmetrization operator as
\begin{equation}
\label{eq:symmetrization}
(\sym(i_1,\ldots,i_k)f)_{i_1,\ldots,i_kj_1,\ldots,j_\ell}:=\frac{1}{k!}\sum_{\sigma}f_{i_{\sigma(1)},\ldots,i_{\sigma(k)}j_1,\ldots,j_\ell},
\end{equation}
where $\sigma$ runs over all permutations of $(1,2,\cdots,k)$. This
operator symmetrizes $f$ with respect to the first $k$ indices. We define
the symmetrization operator $\sym(j_1,\ldots,j_\ell)$, for the last
$\ell$ indices analogously.

We introduce a \textit{first} operator  $\lambda$ the image of which is contained
in the kernel of $L_{k,\ell}$. The operator $\lambda:S^{k-1} M \times S^{\ell-1} M
\to S^{k} M \times S^{\ell} M$ is defined by
\begin{equation}
\label{eq:map_lambda}
(\lambda w)_{i_1,\ldots,i_kj_1,\ldots,j_\ell}:= \sym(i_1,\ldots, i_k)\sym(j_1,\ldots, j_\ell)(g_{i_1j_1}w_{i_2,\ldots,i_kj_2,\ldots,j_\ell}).
\end{equation}
Using \eqref{eq:symmetrization} and $\eqref{eq:map_lambda}$ it is straightforward to verify that
\begin{equation}
\label{eq:map_lambda_kernel}
(\lambda w)_{i_1,\ldots,i_kj_1,\ldots,j_\ell}v^{i_1}\cdots v^{i_k}(v^\perp)^{j_1}\cdots (v^\perp)^{j_\ell}=0,  \quad v \in TM,
\end{equation}
where $v^\perp$ is any vector orthogonal to $v$. Therefore  \eqref{eq:map_lambda_kernel} implies that
\[
   \operatorname{Im}(\lambda) \subset \ker(L_{k,\ell}).
\]

We use the notation $u_{i_1,\ldots,i_k;h},$ for the ($h$) component
functions of the covariant derivative $\nabla u$ of the tensor field
$u$. We define the \textit{second} operator, $d'$ say, by the formula,
\begin{equation}
\label{eq:map_d'}
d':S^{k-1}M\times S^{\ell }M \to S^{k}M\times S^{\ell}M, \quad (d'u)_{i_1,\ldots,i_kj_1,\ldots,j_\ell}:=\sym(i_1,\ldots,i_k) u_{i_2,\ldots,i_kj_1,\ldots,j_\ell;i_1}.
\end{equation}
Then the following holds for any $u\in S^{k-1}M\times S^{\ell }M$,
\begin{multline}
\frac{d}{dt} \bigg(u_{i_1,\ldots,i_{k-1}j_1,\ldots, j_\ell}(\gamma(t))\dot\gamma(t)^{i_1}\cdots\dot\gamma(t)^{i_{k-1}}\eta(t)^{j_1} \cdots\eta(t)^{j_\ell} \bigg)
\\
=(d'u)_{i_1,\ldots,i_{k}j_1,\ldots, j_\ell}\dot\gamma(t)^{i_1}\cdots\dot\gamma(t)^{i_{k}}\eta(t)^{j_1} \cdots\eta(t)^{j_\ell}.
\end{multline}
If  $u|_{\p M}=0$, then $L_{k,\ell}(d'u)=0$ by the fundamental theorem of calculus. Thus 
\[
\{d'u:\: u\in S^{k-1}M\times S^{\ell }M, \: u|_{\p M}=0\}\subset \ker(L_{k,\ell}).
\]
Our main result shows that the kernel of $L_{k,\ell}$ is spanned by
the images of these two linear operators.
\begin{theorem}
\label{th:main}
Let $(M,g)$ be a simple $2$-dimensional Riemannian manifold. Let $f \in S^k M\times S^\ell M$, $k,\ell\geq 1$. Then
\[
L_{k,\ell}f(x,v)=0,\quad\quad(x,v)\in\p_{in}(SM)
\]
if and only if
\[
f=d'u+\lambda w, \quad u \in S^{k-1} M\times S^\ell M, \: u|_{\p M}=0, \quad w\in S^{k-1} M\times S^{\ell-1} M.
\]
\end{theorem}


The key observation needed to prove this theorem is that the mixed ray
transform and the geodesic ray transform can be transformed to one
another, for arbitrary $k,\ell\geq 1$, if $(M,g)$ is a $2$-dimensional
simple Riemannian manifold. A similar observation has already been
obtained for the transverse ray transform by Sharafutdinov
\cite[Chapter 5]{Shara}. The work by Paternain, Salo and Uhlmann
\cite{PSU} proved the s-injectivity of the geodesic ray transform on
simple manifolds in dimension $2$. In Theorem \ref{th:main}, we
characterize the kernel of $L_{k,\ell}$ using their results.

\section{Relation with elastic \textit{qS}-wave tomography}

We describe a mixed ray transform arising from elastic wave tomography. We follow the presentation in \cite[Chapter 7]{Shara}, wherein one can find more details. Let $(x^1,x^2)$ be any curvilinear coordinate system in $\mathbb{R}^2$, where the Euclidean metric is
$$
ds^2=g_{jk}dx^jdx^k.
$$

The elastic wave equations
\begin{equation}\label{elasticw}
\rho\frac{\partial^2 u_j}{\partial t^2}=\sigma_{jk;}^{\quad k}:=\sigma_{jk;l}g^{kl}
\end{equation}
describes the waves traveling in a two-dimensional elastic body $M\subset \R^2$. Here $u(x,t)=(u^1,u^2)$ is the displacement vector. The strain tensor is given by
$$
\varepsilon_{jk}=\frac{1}{2}(u_{j;k}+u_{k;j}),
$$
while the stress tensor is
$$
\sigma_{jk}=C_{jklm}\varepsilon^{lm},
$$
where $\mathbf{C}(x)=(C_{jklm})$ is the elastic tensor and $\rho(x)$ is the density of mass. Here $\varepsilon^{lm}$ is obtained by raising indices with respect to the metric $g_{jk}$. The elastic tensor has the following symmetry properties
\begin{equation}\label{Csymmetry}
C_{jklm}=C_{kjlm}=C_{lmjk}.
\end{equation}

We assume that the elastic tensor is weakly anisotropic, that is, it can be represented as
$$
C_{jklm}=\lambda g_{jk}g_{lm}+\mu(g_{jl}g_{km}+g_{jm}g_{kl})+\delta c_{jklm},
$$
where $\lambda$ and $\mu$ are positive functions  called the
Lam\'e parameters, and $\mathbf{c}=(c_{jklm})$ is an anisotropic
perturbation. Here, $\delta$ is a small positive real number. We note
here that $\mathbf{c}=0$ corresponds to an isotropic medium.

We construct geometric optics solutions to system $(\ref{elasticw})$ using the parameter $\omega=\omega_0/\delta$, 
\[
u_j=e^{i\omega\iota}\sum_{m=0}^\infty\frac{{u}_j^{(m)}}{(i\omega)^m},\quad\varepsilon_{jk}=e^{i\omega\iota}\sum_{m=-1}^\infty\frac{{\varepsilon}_{jk}^{(m)}}{(i\omega)^m},\quad\sigma_{jk}=e^{i\omega\iota}\sum_{m=-1}^\infty\frac{{\sigma}_{jk}^{(m)}}{(i\omega)^m},
\]
where $\iota(x)$ is a real function. 

We substitute the above solutions into equation $(\ref{elasticw})$, assume $u^{(-1)}=\varepsilon^{(-2)}=\sigma^{(-2)}=0$ and equate the terms of the order $-2$ and $-1$ respectively in $\omega$, to obtain
$$
(\lambda+\mu)\langle {u^{(0)}},\nabla\iota\rangle_g\nabla\iota+(\mu\|\nabla\iota\|_g^2-\rho){u^{(0)}}=0.
$$
If we take
\begin{equation}\label{eikonal}
\|\nabla\iota\|^2_g=\frac{\rho}{\mu},
\end{equation}
then 
$$
\langle {u^{(0)}},\nabla\iota\rangle_g=0.
$$
The solutions $u_j^{(0)}$ represent shear waves ($S$-waves), and the
displacement vector $u^{(0)}$ is orthogonal to $\nabla\iota$. We denote
$n_s=\rho/\mu$ and $v_s=1/n_s$. The characteristics of the eikonal
equation $(\ref{eikonal})$ are geodesics of the Riemannian metric
$n_s^2ds^2=n_s^2g_{jk}dx^jdx^k$.

We choose a geodesic $\gamma$ of metric $n_s^2ds^2$
and apply the change of variables,
$$
{u}_j^{(0)}=A_sn_s^{-1}\zeta_j,
$$
where
$$
A_s=\frac{C}{\sqrt{J\rho v_s}}, \quad J^2=n_s^2\det(g_{jk}), \quad C \hbox{ is a constant}.
$$
Then it is shown in \cite[Section 7.1.5.]{Shara} that $\zeta$ satisfies the following \textit{Rytov's law}
\begin{equation}\label{Rytov}
\left(\frac{D\zeta}{\mathrm{d}\iota}\right)_j=-i\frac{1}{\rho v_s^6}(\delta_j^q-\dot{\gamma}_j\dot{\gamma}^q)\omega_0c_{qklm}\dot{\gamma}^k\dot{\gamma}^m\zeta^l,
\end{equation}
where $\frac{D}{d\iota}
$ is the covariant derivative along $\gamma$. We note that $c_{qklm}\dot{\gamma}^k\dot{\gamma}^m$ is quadratic in $\dot{\gamma}$, and symmetric in $k,m$, so the solution $\zeta$ of \eqref{Rytov} depends only on the symmetrization 
$$
f_{jklm}=-i\frac{1}{4\rho v_s^6}(c_{jlkm}+c_{jmkl}).
$$

We assume that for every unit speed geodesic $\gamma:[a,b]\rightarrow M$ (in Riemannian manifold $(M,n_s^2ds^2)$) with endpoints in $\partial M$, the value $\zeta(b)$ of a solution to equation $(\ref{Rytov})$ is known as $\zeta(b)=U(\gamma)\zeta(a)$, where $U(\gamma)$ is the solution operator of \eqref{Rytov} and $\eta(a)$ is the initial value. We formulate an inverse problem.
\begin{problem}
Determine tensor field $f$ from $U(\gamma)$. 
\end{problem}
We linearize this problem as in \cite[Chapter 5]{Shara}.
Take a unit vector $\xi(t)\perp\dot{\gamma}(t)$, which is also parallel along $\gamma$. Then $e_1(t)=\xi(t)$ and $e_2(t)=\dot{\gamma}(t)$ form an orthonormal frame along $\gamma$. In this basis, equation $(\ref{Rytov})$ is
 \begin{equation}\label{Rytov1}
 \begin{split}
 \dot{\zeta}_1=-i\frac{1}{\rho v_s^6}\omega_0c_{1l1m}\dot{\gamma}^l\dot{\gamma}^m\zeta^1, \quad \dot{\zeta}_2=0.
  \end{split}
 \end{equation}
We denote $F(t)=-i\frac{1}{\rho v_s^6}\omega_0c_{1l1m}(\gamma(t))\dot{\gamma}^l(t)\dot{\gamma}^m(t)$. Since \eqref{Rytov1} is a separable first order ordinary differential equation, its solution is
 $$
 \zeta_1(b)=e^{\int_a^bF(t)dt}\zeta_1(a).
 $$
%
 
We take the first-order Taylor expansion of the right-hand side of the equation above   to obtain
 $$
  \zeta_1(b)- \zeta_1(a)\sim\int_a^bF(t)\zeta^1(a) dt.
 $$
 Multiplying this equation by $\eta^1(a)$, we get
\begin{equation}
\label{eq:linearized_problem}
  (\zeta_1(b)- \zeta_1(a))\zeta^1(a)\sim\int_a^bF(t)\zeta^1(a) \zeta^1(a)dt= \int_a^b \omega_0f_{11lm}(\gamma(t))\zeta^1(a)\zeta^1(a)\dot{\gamma}^l(t)\dot{\gamma}^m(t)dt.
\end{equation}
We denote the vector field $\eta(t)=\zeta^i(a)e_i(t), \: \zeta^2(a)=0$, and observe that it is parallel along $\gamma$ and perpendicular to $\dot{\gamma}(t)$. The right-hand side of \eqref{eq:linearized_problem} then takes the form
$$
\int_a^b \omega_0f_{11lm}(\gamma(t))\eta^1(t)\eta^1(t)\dot{\gamma}^l(t)\dot{\gamma}^m(t)dt,
$$
We arrive at the inverse problem. 
\begin{problem}
\label{IP}
Determine the tensor field $f$ from
$$
L_{2,2}(f)=\int_{a}^bf_{jklm}(\gamma(t))\eta^j(t)\eta^k(t)\dot{\gamma}^l(t)\dot{\gamma}^m(t)dt
$$
for all $\gamma$ and $\eta\perp\gamma$, where $\eta$ is parallel along
$\gamma$. 
\end{problem}

\begin{remark}
The tensor field $f$ possesses the same symmetry properties $(\ref{Csymmetry})$ as $\mathbf{C}$. Therefore $f\in S^2M\times S^2 M$. Since 
\[
L_{2,2}(f+d'u+\lambda w)=L_{2,2}(f), \quad \hbox{for any }u\in  S^1M\times S^2 M, \: w\in  S^1M\times S^1 M,
\]
we can only recover the tensor $f$ up to  the kernel of $L_{2,2}$. Thus the Inverse Problem \ref{IP} is a special case of Theorem \ref{th:main}.  
\end{remark}

\section{Context and previous work}

We note that if $\ell=0$ in (\ref{eq:mixed_ray_trans}), the operator
$L_{k,0}$ is the geodesic ray transform $I_k$ for a symmetric
$k$-tensor $f$.  It is well known that $\sym(i_1,\ldots,i_k)\nabla u$
is in the kernel of $I_k$, where $u$ is a symmetric $(k-1)$-tensor
with $u\vert_{\partial\Omega}=0$. If $I_kf=0$ implies
$f=\sym(i_1,\ldots,i_k)\nabla u$, we say $I_k$ is s-injective.

When $(M,g)$ is a $2$-dimensional simple manifold, Paternain, Salo and
Uhlmann \cite{PSU} proved the s-injectivity of $I_k$ for arbitrary
$k$. The standard way to prove s-injectivity of $I_0$ and $I_1$ is to use an energy identity known as the Pestov identity. If $k\geq 2$ this identity alone is not sufficient to prove the s-injectivity. The special case $k=2$ was proved earlier \cite{Shara1} using the proof for boundary rigidity \cite{PU2}.

In dimension three or higher, it has been proved that $I_0$ is
injective \cite{Muk2,Muk1}, and $I_1$ is s-injective \cite{AR}. The
s-injectivity of $I_k$ for $k\geq 2$ is still open for simple
Riemannian manifolds. Under certain curvature conditions, the
s-injectivity of $I_k,\,k\geq 2$ has been proved in \cite{Dair,
  Pestov, PS, Shara}. Without any curvature condition, it has been
proved that $I_2$ has a finite-dimensional kernel \cite{SU4}. If $g$
is in a certain open and dense subset of simple metrics in $C^r,
r\gg1$, containing analytic metrics, the s-injectivity is proved by
analytic microlocal analysis for $k=2$ \cite{SU}. Under a different
assumption that $M$ can be foliated by strictly convex hypersurfaces,
the s-injectivity has been established for $m=0$ \cite{UV}, and
$m=1,2$ \cite{SUV2}.
 
The mixed ray transform ($\ell\neq 0,\,k\neq 0$) is not studied as
extensively as the geodesic ray transform.  In dimension two or
higher, a result similar to Theorem~\ref{th:main} has been obtained
under a restrictive curvature condition \cite{Shara}.
 
When $k=0$, $L_{0,\ell}$ is called the transverse ray transform, also
denoted by $J_\ell$. For $J_\ell$, the situations are quite different
for dimension two and higher dimensions. In dimension three or higher,
$J_\ell$ is injective for $\ell< \dim M$ under certain curvature
conditions \cite{Shara}. However, $J_\ell$ has a nontrivial kernel in
dimension $2$. This problem is related to \textit{polarization}
tomography, for which some results are given under different conditions \cite{Holman, novikov, paternain1605geodesic}.

\section{Proof of Theorem~\ref{th:main}}

Since $(M,g)$ is a $2$-dimensional simple Riemannian manifold, there
exists a diffeomorphism $\phi$ from $M$ onto a closed unit disc
$\overline{\mathbb{D}} $ of $\R^2$. If $g'$ is the pullback of metric
$g$ under $\phi^{-1}$ on $\overline{\mathbb{D}}$ then $g'$ is
conformally Euclidean, meaning that there exists a change of
coordinates after which $g'=he$, where $h$ is some positive function
and $e$ is the Euclidean metric; this was shown in \cite[Theorem
  4]{ahlfors1955conformality} and \cite[ Proposition
  1.3]{Sy}. Therefore there exists global isothermal coordinates
$(x_1,x_2)$ on $M$, so that the metric $g$ can be written as
$e^{2\alpha(x)}(\mathrm{d}x_1^2+\mathrm{d}x_2^2)$ where $\alpha(x)$ is
a smooth real-valued function of $x$.

The global isothermal coordinate structure makes it possible to define
a smooth rotation,
\[
\sigma:TM \to TM, \quad \sigma(v):=(v_2,-v_1),
\]
where $v=(v_1,v_2)$ in these coordinates. This map satisfies
\begin{equation}
\label{eq:prop_of_sigma}
v \perp \sigma(v) \quad \hbox{ and } \quad  \|v\|_g=\|\sigma(v)\|_g.
\end{equation}
Moreover, there exists a linear map
\begin{equation}
\label{eq:mixed_tensor}
\Phi:S^kM \times S^\ell M \to C^\infty(SM), \quad (\Phi f)(x,v):=f_{i_1,\ldots,i_kj_1,\ldots, j_\ell}(x)v^{i_1}\cdots v^{i_k}\sigma(v)^{j_1}\cdots \sigma(v)^{j_\ell}.
\end{equation}
Thus each tensor field $f\in S^k M \times S^\ell M$ is related to a
smooth function on $SM$ via \eqref{eq:mixed_tensor}. We
note that $\Phi$ is not one-to-one since $\Phi(\lambda w)=0$ for any
$w\in S^{k-1} M \times S^{\ell-1} M$, where $\lambda$ is as in
\eqref{eq:map_lambda}. We have the following

\begin{lemma}
For any $f \in S^kM \times S^\ell M $ it holds that
\begin{equation}
\label{eq:map_prop_of_L}
L_{k,\ell} f(x,v)=\int_{0}^{\tau(x,v)} (\Phi f)(\gamma_{x,v}(t),\dot \gamma_{x,v}(t))\mathrm{d}t, \quad\quad(x,v)\in\p_{in}(SM)
\end{equation}
and
\[
L_{k,\ell}:S^kM \times S^\ell M \to C^{\infty}(\p_{in} SM),
\]
if we assume that 
\[
\eta(0)=\sigma(v), \quad (x,v)\in \p_{in}(SM).
\]
\end{lemma}
\begin{proof}

Let $(x,v)\in \p_{in}SM$. We define $\eta=\sigma(v).$ Let $P_t(\eta)$
be the parallel transport of $\eta$ from $T_xM$ to
$T_{\gamma_{x,v}(t)}M$, $t \in [0,\tau(x,v)]$. By the property of
parallel translation, $P_t:T_xM \to T_{\gamma_{x,v}(t)}M$ is an
isometry, whence $\|P_t\eta\|_g=1$ and $\langle
P_t\eta,\dot{\gamma}(t)\rangle_g=0$. Since $M$ is $2$-dimensional, the
continuity of $P_t\eta$ in $t$ with \eqref{eq:prop_of_sigma} imply
\[
P_t\eta=\sigma(\dot\gamma_{x,v}(t)).
\] 

Because the functions $\Phi f$ and $\tau$ are smooth in $\p_{in} (SM)$, the function $L_{k,\ell}(f)$ is smooth in $\p_{in} (SM)$ due to \eqref{eq:map_prop_of_L}.
\end{proof}

\noindent
Let $f \in S^k M\times S^\ell M$. Simplifying the notation, from here on we do not distinguish tensor $f$ from function $\Phi(f)$. We notice first that
\begin{equation}
\label{eq:observation_1}
f(x,v)=(-1)^{\ell -N(j_1,\ldots, j_\ell)}f_{i_1,\ldots,i_kj_1,\ldots, j_\ell}(x)v^{i_1}\cdots v^{i_k}v_1^{\ell- N(j_1,\ldots, j_\ell)}v_2^{{N(j_1,\ldots, j_\ell)}}, \quad (x,v)\in SM,
\end{equation}
where ${N(j_1,\ldots, j_\ell)}$ is the number of $1$s in $(j_1,\ldots,
j_\ell)$. We let $\delta$ be the map that maps $1$s in $(j_1,\ldots,
j_\ell)$ to $2$s and vice versa. We denote by $\delta(j_1,\ldots,
j_\ell)$ the $\ell$-tuple obtained from applying $\delta$ to
$(j_1,\ldots, j_\ell)$.  Then we define a linear operator
\begin{equation}
\label{eq:A_def}
A:S^k M \times S^\ell M \to S^k M \times S^\ell M, \quad
(Af)_{i_1,\ldots,i_kj_1,\ldots, j_\ell}=(-1)^{\ell -N(j_1,\ldots, j_\ell)}f_{i_1,\ldots,i_k\delta(j_1,\ldots, j_\ell)}.
\end{equation}
We note that if $\ell=1,$ then $A$ and the Hodge star operator coincide. Formula \eqref{eq:A_def} implies that $A$ is invertible with the following inverse
\begin{equation}
\label{eq:A_inverse}
A^{-1}=(-1)^\ell A.
\end{equation}
We then point out that
\begin{equation}
\label{eq:symmetric}
(Af)_{i_1,\ldots ,i_kj_1,\ldots, j_\ell}(x)v^{i_1}\ldots v^{i_k}v^{j_1}\cdots v^{j_\ell}=(\hbox{Sym} Af)_{i_1,\ldots i_kj_1,\ldots, j_\ell}(x)v^{i_1}\ldots v^{i_k}v^{j_1}\cdots v^{j_\ell}.
\end{equation}
The notation Sym$h$ stands for the full symmetrization of the tensor field $h$.
 
\medskip

Using equations \eqref{eq:observation_1}, \eqref{eq:A_def} and
\eqref{eq:symmetric}, we find that
\begin{equation}
\label{eq:mixed_and_ray_trans}
L_{k,\ell}(f)=I_{k+\ell}(\sym(Af)),
\end{equation}
where $I_{k+\ell}$ is the geodesic ray transform on symmetric tensor field $h \in S^{k+\ell}M$, defined by the formula
$$
I_{k+\ell}(h)(x,v)=\int_{0}^{\tau(x,v)}h_{i_1,\ldots,i_{k+\ell}}(\gamma_{x,v}(t))\dot\gamma_{x,v}(t)^{i_1}\cdots \dot\gamma_{x,v}(t)^{i_{k+\ell}}\mathrm{d}t ,\quad\quad(x,v)\in\p_{in}(SM).
$$
By \eqref{eq:mixed_and_ray_trans} and  \cite[Theorem 1.1]{PSU} it holds that for any $h\in S^kM \times S^\ell M$,
\begin{equation}
\label{eq:kernel_of_ray_trans}
L_{k,\ell}(h)=0 \hbox{ if and only if } \hbox{Sym}Ah=d^sv,  \quad v \in S^{k+\ell-1}M, \quad v|_{\p M}=0.
\end{equation}
In the above, $d^s$ stands for the inner derivative, that is, the symmetrization of the covariant derivative
\begin{equation}
\label{eq:d^s}
d^su=\sym(\nabla u),  \quad u \in S^{k+\ell-1}M.
\end{equation}

If $L_{k,\ell}(f)=0$ then, with \eqref{eq:A_inverse} and \eqref{eq:kernel_of_ray_trans}, we can write  
\[
f=(-1)^{\ell}A(\hbox{Sym}(Af)+(Af-\hbox{Sym}(Af)))=(-1)^{\ell}A(d^s u)+f+(-1)^{\ell+1}A(\hbox{Sym}(Af)).
\]
We conclude that the claim of Theorem \ref{th:main} holds if 
\[f+(-1)^{\ell+1}A(\sym(Af))=\lambda w, \quad  A(d^su-d'u)= \lambda w', \quad d'A=Ad' ,
\]
for some  $w,w' \in S^{k-1}M\times S^{\ell -1}M$ and $u \in S^{k+\ell-1}M$. These equations will be proved in the following subsections.

\subsection{Analysis of operator $A\sym A$}

In this subsection, we prove the following identity for any $f\in S^{k} M \times S^{\ell} M$:
\begin{equation}
\label{eq:f-AsymA}
f+(-1)^{\ell+1}A(\sym(Af))=\lambda w \quad \hbox{ for some }
w \in S^{k-1} M \times S^{\ell-1} M.
\end{equation}
We start with a lemma that characterizes the kernel of $A\sym A$

\begin{lemma}
\label{Le:Ker_of_AsymA}
For the linear maps $A\emph{\sym} A
:S^{k} M \times S^{\ell}M \to S^{k} M \times S^{\ell}M 
$ and 
\\
$\lambda:S^{k-1} M \times S^{\ell-1}M \to S^{k} M \times S^{\ell}M 
$ the following holds
\[
\ker(A \emph{\sym} A) = \emph{Im}(\lambda).
\]
\end{lemma}
\begin{proof}
We use the notation $\otimes_s$ for the symmetric product of tensors. We note that operator $A$ maps a basis element $\big((\otimes^{h}dx^1) \otimes_s(\otimes^{k-h} dx^2)\big)\otimes \big((\otimes^{a}dx^1) \otimes_s(\otimes^{\ell-a} dx^2)\big),  \: h\in \{0,\ldots, k\}, a \in \{0,\ldots, \ell\}$ of $S^k M \times S^\ell M$ to
\[
(-1)^{\ell-a}\big((\otimes^{h}dx^1) \otimes_s(\otimes^{k-h} dx^2)\big)\otimes \big((\otimes^{\ell-a}dx^1) \otimes_s(\otimes^{a} dx^2)\big).
\]
We also note that the choice of isothermal coordinates implies
\begin{equation}
\label{eq:lambda_isothermal}
\lambda (a\otimes b)=e^{2\alpha(x)}\big((dx^1 \otimes_s a) \otimes (dx^1 \otimes_s b)+ (dx^2 \otimes_s a) \otimes ( dx^2\otimes_s b) \big), \quad a\otimes b\in S^{k-1} M \times S^{\ell-1}M.
\end{equation}
Since $A$ is a bijection, it suffices to prove
\begin{equation}
\label{eq:inclusions}
   \operatorname{Im}(\lambda) =\ker (\sym A).
\end{equation}
We prove first that $\operatorname{Im}(\lambda) \subset\ker (\sym A)$. 
In view of the linearity of $\lambda$, it suffices to prove that $\lambda w\in \ker \sym A$ when 
\[
w=r(x)\big((\otimes^{h-1}dx^1) \otimes_s(\otimes^{k-h} dx^2)\big)\otimes \big((\otimes^{a-1}dx^1) \otimes_s(\otimes^{\ell-a} dx^2)\big), \quad h\in \{1,\ldots, k\}, a \in \{1,\ldots, \ell\}.
\]
Then
\begin{multline}
e^{-2\alpha(x)}A \lambda w= (-1)^{\ell-a}r(x)\bigg(\big((\otimes^{h}dx^1) \otimes_s(\otimes^{k-h} dx^2)\big)\otimes \big((\otimes^{\ell-a}dx^1)\otimes_s(\otimes^{a} dx^2)\big) 
\\
- \big((\otimes^{h-1}dx^1) \otimes_s(\otimes^{k-h+1} dx^2)\big)\otimes \big((\otimes^{\ell-a+1}dx^1) \otimes_s(\otimes^{a-1} dx^2)\big)\bigg).
\end{multline}
Since $\sym$ is a linear operator, we have $\sym A(\lambda
w)=0$. Therefore $\operatorname{Im}(\lambda) \subset \ker(\sym A)$

Now we prove that $\ker (\sym A) \subset  \operatorname{Im}(\lambda)$. 
We assume first that $f=\sum_{m=1}^M u_m$, where 
\begin{equation}
\label{eq:first_case}
u_m = r_m(x)\big((\otimes^{h}dx^1) \otimes_s(\otimes^{k-h} dx^2)\big)\otimes \big((\otimes^{\ell-a}dx^1) \otimes_s(\otimes^{a} dx^2)\big), \quad h+a \leq \min\{k,\ell\}.
\end{equation}
Then we can write $f=\sum_{H=0}^{k+\ell}f_H,$ where $f_H=0,$ if $H\geq \min\{k,\ell\}$ and otherwise 
\[
\begin{split}
f_H=\sum_{h=0}^{H}a_{H,h} f_{H,h}, \quad f_{H,h}:=\big((\otimes^{h}dx^1) \otimes_s(\otimes^{k-h} dx^2)\big)\otimes \big((\otimes^{\ell-(H-h)}dx^1) \otimes_s(\otimes^{H-h} dx^2)\big).
\end{split}
\]  
Moreover $f \in \ker (\sym A)$ if and only if $f_H \in \ker (\sym A)$ for every $H\in \{1,\ldots, \min\{k,\ell\}\}$. In the following we study the tensor $f_H$, for a  given  $H\in \{1,\ldots, \min\{k,\ell\}\}$. 

\medskip

For $h \in \{1,\ldots, H\}$ we define $w_h\in S^{k-1} M \times S^{\ell-1}M$ by formula
\[
w_h=\big((\otimes^{h-1}dx^1) \otimes_s(\otimes^{k-h} dx^2)\big)\otimes \big((\otimes^{\ell-(H-h+1)}dx^1) \otimes_s(\otimes^{H-h} dx^2)\big).
\]
Then \eqref{eq:lambda_isothermal} yields
\[
\begin{split}
\lambda w_h
=&e^{2\alpha(x)}(f_{H,h}+f_{H,h-1}).
\end{split}
\]
This implies the recursive formula
\[
f_{H,h}=\lambda (e^{-2\alpha(x)}w_h) -f_{H,h-1}.
\]
Thus for every $h \in \{0,\ldots, H\}$ there exists $w'_h \in S^{k-1} M \times S^{\ell-1}M$  such that
\begin{equation}
\label{eq:recursive_2}
f_{H,h}=\lambda w'_h +(-1)^{h}f_{H,0}.
\end{equation}
Therefore there exists $w_H \in S^{k-1} M \times S^{\ell-1}M$ such that
\[
f_H=\sum_{h=0}^{H}a_{H,h}f_{H,h}=\lambda w_H+ f_{H,0}\sum_{h=0}^{H} (-1)^{h}a_{H,h}.
\]

If $f \in \ker \sym A$  it holds by the first part of this proof that
\[
\sym A f_H=(\sym Af_{H,0})\bigg(\sum_{i=0}^{H}(-1)^{h}a_{H,h}\bigg)=0.
\]
Since $\sym Af_{H,0}\neq 0$ it follows that $\sum_{i=0}^{H}(-1)^{h}a_{H,h}=0$ whence $f_H=\lambda w_H$. This implies $f=\lambda w$ for some $w\in S^{k-1} M \times S^{\ell-1}M$.

\medskip
If $f \in \ker\sym A$ and we cannot write $f=\sum_{m=1}^Mu_m$, where each $u_m$ satisfies \eqref{eq:first_case}, then there exists $u_m$ that satisfies
\[
(\otimes^{h}dx^1) \otimes_s(\otimes^{k-h} dx^2)\big)\otimes \big((\otimes^{\ell-a}dx^1) \otimes_s(\otimes^{a} dx^2), \quad \min \{k,\ell\} < h+a \leq  \max \{k,\ell\}.
\] 
Therefore $f_{H}\neq 0$ for some $\min \{k,\ell\} <H \leq  \max \{k,\ell\}$ and there exist two sub cases. If $k< H \leq \ell$, then
\[
f_H=\sum_{h=0}^{k} a_{H,h}f_{H,h}, \quad f_{H,h}=\big((\otimes^{h}dx^1) \otimes_s(\otimes^{k-h} dx^2)\big)\otimes \big((\otimes^{\ell-(H-h)}dx^1) \otimes_s(\otimes^{H-h} dx^2)\big).
\]
If $\ell < H \leq k$, then
\[
f_H=\sum_{h=0}^{\ell} a_{H,h}f_{H,h}, \quad f_{H,h}=\big((\otimes^{H-\ell+h}dx^1) \otimes_s(\otimes^{k-h-H+\ell} dx^2)\big)\otimes \big((\otimes^{h}dx^1) \otimes_s(\otimes^{\ell-h} dx^2)\big).
\]
By an analogous recursive argument as before, we find that $f=\lambda
w$, for some $w \in S^{k-1} M \times S^{\ell-1}M$. This completes the
proof. \end{proof}

By the proof of the previous Lemma we can write any $f \in S^{k} M \times S^{\ell}M$ in the form
\begin{equation}
\label{eq:rep_of_f}
f=\lambda w + \sum_{H=0}^{k+\ell} r_H f_{H,0}, \quad r_H\in C^\infty(M),
\end{equation}
for some $w \in S^{k-1} M \times S^{\ell-1}M $. 
Next, we prove that
\begin{equation}
\label{eq:AsymA_f_0}
A\sym A f_{H,0}=(-1)^{\ell}f_{H,0}+\lambda w, \quad H \in \{1,\ldots, k+\ell\}.
\end{equation}
We assume first that $H\leq \min\{k,\ell\}$. Then 
\[
f_{H,0}=\big(\otimes^{k} dx^2\big)\otimes \big((\otimes^{\ell-H}dx^1) \otimes_s(\otimes^{H} dx^2)\big).
\]
This implies
\[
\begin{split}
\sym A f_{H,0}=&(-1)^{\ell}(\otimes^{H}dx^1\otimes_s(\otimes^{k+\ell-H}dx^2))
\\
=&(-1)^{\ell}\frac{1}{(k+\ell)!}\sum_{h=0}^{H}A_h(\otimes^{h}dx^1\otimes_s(\otimes^{k-h}dx^2)) \otimes (\otimes^{H-h}dx^1\otimes_s(\otimes^{\ell-H+h}dx^2)),
\end{split}
\]
where $\sum_{h=0}^{H}A_h=(k+\ell)!$. Using \eqref{eq:recursive_2} we obtain
\[
\begin{split}
A\sym A f_{H,0}
=&(-1)^{\ell}\frac{1}{(k+\ell)!}\sum_{h=0}^{H}(-1)^{h}A_hf_{H,h}=(-1)^{\ell}\frac{1}{(k+\ell)!}\bigg(\sum_{h=0}^{H}A_h\bigg)f_{H,0} +\lambda w
\\
=&(-1)^{\ell}f_{H,0}+\lambda w.
\end{split}
\]
If $\min\{k,\ell\} < H \leq \max\{k,\ell\}$ it follows by a similar
argument that $A\sym A f_{H,0}=(-1)^{\ell}f_{H,0}+\lambda
w$. Therefore, we proved \eqref{eq:AsymA_f_0}.

\medskip\medskip

\noindent
Equation \eqref{eq:f-AsymA} follows from Lemma~\ref{Le:Ker_of_AsymA}
and \eqref{eq:rep_of_f}--\eqref{eq:AsymA_f_0}.

\subsection{Analysis of operator $A d^s$}

We note that $S^{k+\ell}M\subset S^{k}M\times S^{\ell}M$. Therefore, we
can extend the inner derivative, $d^s$, to an operator $d^s:S^{k-1}M\times
S^{\ell}M \to S^{k}M\times S^{\ell}M$ and evaluate $d^s-d'$. In this
subsection, we show that for any $u \in S^{k-1}M\times S^{\ell}M $ the
following equations hold,
\begin{align} \label{eq:A(d^s-d')}
 A(d^su-d'u) &= \lambda w \quad \hbox{ for some }
 w \in S^{k-1}M\times S^{\ell -1}M; 
\\
\label{eq:Ad'=d'A}
 d'A &=Ad'.
\end{align}
Since $Ad^s$ and $A d'$ are linear it suffices to prove the claims for 
\[
u=r(x)\big((\otimes^{h-1}dx^1) \otimes_s(\otimes^{k-h} dx^2)\big)\otimes \big((\otimes^{a}dx^1) \otimes_s(\otimes^{\ell-a} dx^2)\big), \quad r\in C^{\infty}(M).
\]
By \eqref{eq:map_d'} and \eqref{eq:A_def} we have
\begin{multline} \label{eq:Ad'u}
Ad'u= (-1)^{\ell-a}\bigg(\bigg(\frac{\p}{\p x^1}r(x)-R_1\bigg)\big((\otimes^{h}dx^1) \otimes_s(\otimes^{k-h} dx^2)\big)\otimes \big((\otimes^{\ell-a}dx^1) \otimes_s(\otimes^{a} dx^2)\big)
\\
+ \bigg(\frac{\p}{\p x^2}r(x)-R_2\bigg)\big((\otimes^{h-1}dx^1) \otimes_s(\otimes^{k-h+1} dx^2)\big)\otimes \big((\otimes^{\ell-a}dx^1) \otimes_s(\otimes^{a} dx^2)\big)\bigg),
\end{multline}
where $R_m=\sum_{s=1}^{k+\ell-1}r_{i_1,\ldots,i_{s-1}p,i_{s+1},\ldots, i_{k+\ell}}\Gamma^p_{mi_s}$, $m\in \{1,2\}$ and $r_{i_1,\ldots,i_{s-1}p,i_{s+1},\ldots, i_{k+\ell}}\in \{0,r\}$ depending on $(i_1,\ldots, i_{k+\ell})$.

We write $H=h+a$, assume that $H \leq \min\{k,\ell\}$ and denote
$\widetilde R_m= \frac{\p}{\p x^m}r(x)-R_m$. Then we obtain from
\eqref{eq:A_def} and \eqref{eq:d^s},
\[
\begin{split}
d^su
=&\:\widetilde R_1 \frac{1}{(k+\ell)!} \sum^{H}_{j=0}A_j\big((\otimes^{j}dx^1) \otimes_s(\otimes^{k-j} dx^2)\big)\otimes  \big((\otimes^{H-j}dx^1) \otimes_s(\otimes^{\ell+j-H} dx^2)\big)
\\
+&\:  \widetilde R_2 \frac{1}{(k+\ell)!}\sum^{H-1}_{i=0}B_i\big((\otimes^{i}dx^1) \otimes_s(\otimes^{k-i} dx^2)\big)\otimes  \big((\otimes^{H-i-1}dx^1) \otimes_s(\otimes^{\ell+i-H+1} dx^2)\big),
\end{split}
\]
where $\sum_{j=0}^H A_j=\sum B_{i=0}^{H-1}=(k+\ell)!$. This
yields
\begin{multline}
Ad^su = \widetilde R_1 \frac{1}{(k+\ell)!} \sum^{H}_{j=0}(-1)^{\ell-H+j}A_j\big((\otimes^{j}dx^1) \otimes_s(\otimes^{k-j} dx^2)\big)\otimes  \big((\otimes^{\ell+j-H}dx^1) \otimes_s(\otimes^{H-j} dx^2)\big)
\\
-\widetilde R_2 \frac{1}{(k+\ell)!}\sum^{H-1}_{i=0}(-1)^{\ell-H+i}B_i\big((\otimes^{i}dx^1) \otimes_s(\otimes^{k-i} dx^2)\big)\otimes  \big((\otimes^{\ell+i-H+1}dx^1) \otimes_s(\otimes^{H-i-1} dx^2)\big).
\end{multline}
We define
\[
g_{H,j}=\big((\otimes^{j}dx^1) \otimes_s(\otimes^{k-j} dx^2)\big)\otimes \big((\otimes^{\ell +j-H}dx^1) \otimes_s(\otimes^{H-j} dx^2)\big), \quad j \in \{0,\ldots, H\},
\] 
and
\[
v_{H,j}=\big((\otimes^{j}dx^1) \otimes_s(\otimes^{k-j-1} dx^2)\big)\otimes \big((\otimes^{\ell+j-H}dx^1) \otimes_s(\otimes^{H-j-1} dx^2)\big), \quad j \in \{1,\ldots, H\}.
\]
Then \eqref{eq:lambda_isothermal} implies that $\lambda v_{H,j}=e^{2\alpha(x)}(g_{H,j}+g_{H,j+1})$. We obtain 
\[
g_{H,j}= \lambda w_{H,j} +(-1)^{H-j}g_{H,H}, \quad \hbox{for some } w_j \in S^{k-1}M\times S^{\ell -1}M.
\]
Thus
\[
\begin{split}
dA'u
=&
(-1)^{\ell-a}\bigg(\widetilde R_1g_{H,h}+\widetilde R_2g_{H-1,h-1}\bigg)
\\
=&(-1)^{\ell}\bigg(\widetilde R_1g_{H,H}+\widetilde R_2g_{H-1,H-1}\bigg)+\lambda w', \quad \hbox{ for some } w' \in S^{k-1}M\times S^{\ell-1}M
\end{split}
\]
and
\[
\begin{split}
Ad^su
=& \:\widetilde R_1 \frac{1}{(k+\ell)!} \sum^{H}_{j=0}(-1)^{\ell-H+j}A_j(\lambda w_{j,H} +(-1)^{H-j
}g_{H,H})
\\
+ & \:\widetilde R_2 \frac{1}{(k+\ell)!}\sum^{H-1}_{i=0}(-1)^{\ell-H+i+1}B_i(\lambda w_{i,H-1} +(-1)^{H-1-i
}g_{H-1,H-1}\big)
\\
=&(-1)^{\ell} \bigg(\widetilde R_1 g_{H,H}+ \widetilde R_2g_{H-1,H-1}\bigg)+\lambda w'',  \quad \hbox{ for some } w'' \in S^{k-1}M\times S^{\ell-1}M. 
\end{split}
\]
These identities imply
\[
A(d^su-d'u)=\lambda w, \quad w \in S^{k-1}M\times S^{\ell-1}M.
\]

For the case $\min\{k,\ell\} < H \leq \max\{k,\ell\}$, the proof is
similar and is omitted. Therefore have proved \eqref{eq:A(d^s-d')}.

\medskip\medskip

\noindent
Finally we prove equation \eqref{eq:Ad'=d'A}. We note that 
\begin{align*}
d'Au= \:(-1)^{\ell-a}\bigg(\widetilde R_1\big((\otimes^{h}dx^1) \otimes_s(\otimes^{k-h} dx^2)\big)\otimes \big((\otimes^{\ell-a}dx^1) \otimes_s(\otimes^{a} dx^2)\big)
\\
+ \:\widetilde R_2\big((\otimes^{h-1}dx^1) \otimes_s(\otimes^{k-h+1} dx^2)\big)\otimes \big((\otimes^{\ell-a}dx^1) \otimes_s(\otimes^{a} dx^2)\big)\bigg).
\end{align*}
Thus \eqref{eq:Ad'=d'A} holds since the previous equation coincides
with \eqref{eq:Ad'u}.

\bibliographystyle{abbrv}
\bibliography{biblio}

\begin{thebibliography}{10}

\bibitem{ahlfors1955conformality}
L.~V. Ahlfors.
\newblock Conformality with respect to riemannian metrics. annales academi
  scientiarum fennicae series a. i.
\newblock {\em Mathematica}, 206:1--22, 1955.

\bibitem{AR}
Y.~Anikonov and V.~Romanov.
\newblock On uniqueness of determination of a form of first degree by it
  integrals along geodesics.
\newblock {\em J. Inverse Ill-Posed Probl.}, 5:467--480, 1997.

\bibitem{chapman1992traveltime}
C.~H. Chapman and R.~G. Pratt.
\newblock Traveltime tomography in anisotropic media —- i. theory.
\newblock {\em Geophysical Journal International}, 109(1):1--19, 1992.

\bibitem{Dair}
D.~S. Daribekov.
\newblock Integral geometry problem for nontrapping manifolds.
\newblock {\em Inverse Problems}, 22:431--445, 2006.

\bibitem{Holman}
S.~Holman.
\newblock Generic local uniqueness and stability in polarization tomography.
\newblock {\em Journal of Geometric Analysis}, 23(1):229--269, 2013.

\bibitem{Muk2}
R.~G. Mukhometov.
\newblock On the problem of integral geometry (russian).
\newblock {\em Math. problems of geophysics. Akad. Nauk SSSR, Sibirsk., Otdel.,
  Vychisl., Tsentr, Novosibirsk}, 6, 1975.

\bibitem{Muk1}
R.~G. Mukhometov.
\newblock On a problem of reconstructing {R}iemannian metrics.
\newblock {\em Sibirsk. Mat. Zh.}, 22:119--135, 1987.

\bibitem{novikov}
R.~Novikov and V.~Sharafutdinov.
\newblock On the problem of polarization tomography: I.
\newblock {\em Inverse problems}, 23(3):1229, 2007.

\bibitem{PSU}
G.~Paternain, M.~Salo, and G.~Uhlmann.
\newblock Tensor tomography on simple surfaces.
\newblock {\em Invent. Math.}, 193:229--247, 2013.

\bibitem{paternain1605geodesic}
G.~Paternain, M.~Salo, G.~Uhlmann, and H.~Zhou.
\newblock The geodesic x-ray transform with matrix weights, preprint (2016).
\newblock {\em arXiv preprint arXiv:1605.07894}, 2.

\bibitem{Pestov}
L.~Pestov.
\newblock {\em Well-posedness questions of the ray tomography problems (in
  {R}ussian)}.
\newblock Siberian Science Press, Novosibirsk, 2003.

\bibitem{PS}
L.~Pestov and V.~A. Sharafutdinov.
\newblock Integral geometry of tensor fields on a manifold of negative
  curvature.
\newblock {\em Siberian Math. J.}, 29:427--441, 1988.

\bibitem{PU2}
L.~Pestov and G.~Uhlmann.
\newblock Two dimensional compact simple riemannian manifolds are boundary
  distance rigid.
\newblock {\em Annals of mathematics}, pages 1093--1110, 2005.

\bibitem{Shara1}
V.~Sharafutdinov.
\newblock Variations of dirichlet-to-neumann map and deformation boundary
  rigidity of simple 2-manifolds.
\newblock {\em The Journal of Geometric Analysis}, 17(1):147, 2007.

\bibitem{Shara}
V.~A. Sharafutdinov.
\newblock {\em Integral geometry of tensor fields}, volume~1.
\newblock Walter de Gruyter, 1994.

\bibitem{SU4}
P.~Stefanov and G.~Uhlmann.
\newblock Stability estimates for the x-ray transform of tensor fields and
  boundary rigidity.
\newblock {\em Duke Mathematical Journal}, 123(3):445--467, 2004.

\bibitem{SU}
P.~Stefanov and G.~Uhlmann.
\newblock Boundary ridigity and stability for generic simple metrics.
\newblock {\em Journal of Amer. Math. Soc.}, 18:975--1003, 2005.

\bibitem{SUV2}
P.~Stefanov, G.~Uhlmann, and A.~Vasy.
\newblock Inverting the local geodesic x-ray transform on tensors.
\newblock {\em arXiv preprint arXiv:1410.5145}, 2014.

\bibitem{Sy}
J.~Sylvester.
\newblock An anisotropic inverse boundary value problem.
\newblock {\em Communications on Pure and Applied Mathematics}, 43(2):201--232,
  1990.

\bibitem{UV}
G.~Uhlmann and A.~Vasy.
\newblock The inverse problem for the local geodesic ray transform.
\newblock {\em Invent. Math.}, 205:83--120, 2016.

\end{thebibliography}

\end{document}